\documentclass[11pt,a4paper,reqno,xcolor=dvipsnames,]{amsart}
\usepackage[hmargin=2cm,bmargin=3.2cm,tmargin=3.2cm]{geometry}
\usepackage{enumerate}
\usepackage{amsmath,amsthm,amssymb,amsfonts,latexsym}
\usepackage{mathbbol}

\usepackage{esvect}
\usepackage{booktabs}
\usepackage{faktor}

\usepackage[colorlinks, linkcolor=blue, filecolor=blue,
 citecolor = olive, urlcolor=purple]{hyperref}
\usepackage[gen]{eurosym}
\usepackage[english]{babel}
\usepackage[normalem]{ulem}
\usepackage{latexsym}
\usepackage{orcidlink}
\usepackage{tikz}					
\usetikzlibrary{arrows.meta, decorations.markings,calc}

\DeclareMathOperator{\dd}{\mbox{d}\!} 
\DeclareMathOperator{\Diam}{diam} 
\DeclareMathOperator{\Tri}{tri} 
\DeclareMathOperator{\mGirth}{girth} 
\DeclareMathOperator{\rad}{rad} 
\DeclareMathOperator{\Inr}{inr} 
\DeclareMathOperator{\dist}{dist} 

\DeclareMathOperator{\Vol}{Vol}

\DeclareMathOperator{\Id}{Id}

\DeclareMathOperator{\diag}{diag}

 \def\LGVD{L_{\mG;\mVD}} 
\def\DeltaGVD{\Delta_{\Graph;\mVD}} 
 
  \def\mG{\mathsf{G}}
 
 \def\mV{\mathsf{V}}
 \def\mE{\mathsf{E}}
 
 \def\mK{\mathsf{K}}
 \def\mC{\mathsf{C}}

 \def\mT{\mathsf{T}}
  \def\mv{\mathsf{v}}
 \def\mw{\mathsf{w}}
 \def\mz{\mathsf{z}}
 \def\me{\mathsf{e}}

 \def\mw{\mathsf{w}}
 \def\mz{\mathsf{z}}

\newcommand{\mVD}{{\mV^{\mathrm D}}}
\newcommand{\mVDn}{{\mV^{\mathrm D}_n}}

\newcommand{\Graph}{\mathcal{G}}

\newcommand{\R}{\mathbb{R}}
\newcommand{\N}{\mathbb{N}}

\newcommand{\C}{\mathbb{C}}

\def\XXint#1#2#3{{\setbox0=\hbox{$#1{#2#3}{\int}$ }
\vcenter{\hbox{$#2#3$ }}\kern-.6\wd0}}

\usepackage{aliascnt}

\theoremstyle{plain}
\newtheorem{theo}{Theorem}[section]

\newaliascnt{cor}{theo}
\newaliascnt{prop}{theo}
\newaliascnt{lemma}{theo}
\newaliascnt{conj}{theo}

\newtheorem{lemma}[lemma]{Lemma}
\newtheorem{prop}[prop]{Proposition}
\newtheorem{cor}[cor]{Corollary}
\newtheorem{conj}[conj]{Conjecture}

\aliascntresetthe{cor}
\aliascntresetthe{prop}
\aliascntresetthe{lemma} 
\aliascntresetthe{conj}

\theoremstyle{defi} 

\newaliascnt{defi}{theo}
\newaliascnt{assum}{theo}
\newaliascnt{assums}{theo}
\newaliascnt{prob}{theo}

\newtheorem{defi}[defi]{Definition}

\aliascntresetthe{defi}
\aliascntresetthe{assum}
\aliascntresetthe{assums}
\aliascntresetthe{prob}

\theoremstyle{rem}

\newaliascnt{rems}{theo}
\newaliascnt{rem}{theo}
\newaliascnt{exa}{theo}
\newaliascnt{exs}{theo}

\newtheorem{rem}[rem]{Remark}

\aliascntresetthe{rems}
\aliascntresetthe{rem}
\aliascntresetthe{exa}
\aliascntresetthe{exs}

\numberwithin{equation}{section}
\numberwithin{lemma}{section}

\title[The role of expanders in the spectral geometry of metric graphs]{The role of expanders in the\\ spectral geometry of metric graphs} 

\author[D.~Mugnolo]{Delio Mugnolo
\orcidlink{0000-0001-9405-0874}}

\address{Lehrstuhl Analysis, Fakult\"at Mathematik und Informatik, Fern\-Universit\"at in Hagen, D-58084 Hagen, Germany}
\email{delio.mugnolo@fernuni-hagen.de}

\thanks{I would like to thank Robert J.\ George (Caltech) for suggesting the crucial ChatGPT-aided arguments in the proof of \autoref{rem:no-meandist-bound} and of \autoref{theo:no-inr-bound}; and Luís Baptista (Lisbon), Gregory Berkolaiko (Texas A\&M), James B.\ Kennedy (Aveiro) and Noah Kravitz (Oxford) and Pavel Kurasov (Stockholm) for many interesting discussions. Portions of the manuscript were revised with the assistance of the LLMs Qwen, Gemini and Claude.
}

\begin{document}

\maketitle

\begin{abstract}
Expanders are families of graphs that are sparse in edges but dense in connectivity. 

After reviewing combinatorial and spectral definitions of expanders, we use precise lower bounds on Ramanujan graphs to investigate upper bounds -- and, specifically, the lack thereof -- on the eigenvalues of the Laplacian on \emph{metric} graphs in terms of volume, diameter, girth,  mean distance, and torsional rigidity, among others.
\end{abstract}

\section{Introduction}

Expanders are growing families of (discrete) graphs whose local geometry essentially decouples from their global geometry: even as they become larger and larger, their fine structure forces them to sustain well-mixing diffusive processes. A canonical example of this behavior is that their diameters only grow logarithmically in the number of vertices.

Expander graphs were first explicitly constructed by Gregory Margulis in \cite{Mar73}, who called them \emph{concentrators}. Their spectral characterization was established by Noga Alon and Vitali Milman in~\cite{AloMil85,Alo86} through the discrete Cheeger inequality, laying the foundations of modern spectral graph theory. 
They are nowadays among the most important objects in discrete mathematics and theoretical computer science. By contrast, their analytical aspects have received comparatively less attention, despite their spectacular role in the recent solution of the Kadison--Singer Problem (see \cite{MarSpiSri15b,CasTre06}).

Indeed, expanders also provide a remarkable class of spaces exhibiting rapid heat dissipation and uniform functional inequalities, and extreme spectral-geometric behavior. These analytical properties motivate the study of metric expanders from the viewpoint of spectral geometry, where expanders play the role of fairy godfathers as they allow us to disprove arguably natural conjectures about the rate of convergence of equilibrium in diffusive processes, measured in terms of a metric graph's \emph{spectral gap}.

Let us say that a positive quantity $\Phi(\Graph)$ attached to a metric graph $\Graph$ is a \emph{metric quantity} if it scales linearly in the edge length: this is the case for the  volume $|\Graph|$ and the diameter $\Diam(\Graph)$, but also for more exotic quantities like the mean distance $\rho(\Graph)$, the inradius $\Inr(\Graph;\mVD)$ or the cubic root of the torsional rigidity $T(\Graph;\mVD)$ with respect to a set $\mVD$ of vertices where a Dirichlet condition is imposed. The main -- and simple -- idea of this note is that if we can prove that a metric quantity tends to $\infty$ along a suitably chosen family of expanders, then no upper bound of the spectral gap in terms of the square of said metric quantity is possible; or, if the metric quantity tends to 0, then no lower bound is possible. In other words: they do not represent an impediment to arbitrarily large -- respectively, arbitrarily small -- mixing rates of a diffusive process. In the context of metric graphs these ideas have been, to the best of my knowledge, first developed in \cite{BerKenKur23}.

We first review in \autoref{sec:laplacianseverywhere} known facts from the theory of Laplacians on discrete and metric graphs, with a special focus on Ramanujan graphs, a two-parameter (discrete) family of regular graphs that is known to deliver expanders.  We also offer a reminder of metric graphs, precisely introducing the spectral gap.

In \autoref{sec:failure} we link both environments and show how expander metric graphs can be used, through a fairly transparent mechanism, to disprove geometric upper bounds on the lowest positive Laplacian eigenvalue $\lambda_2(\Graph)$. We here recall known results, offer new insights on old results, and present an AI-assisted proof of a new result, which all share the same basic strategy. Our main result in this section is \autoref{cor:lowerb}, which unifies our theory and shows how to examine the failure of lower bounds, too. Throughout this section, our method is always the same: we fix one parameter of the Ramanujan graphs (corresponding to their degree) and let the second parameter (corresponding to their vertex number) tend to $\infty$. 

In~\autoref{sec:dirichlet} we turn our focus to metric graphs with Dirichlet boundary conditions and offer two contributions to the study of the lowest positive Laplacian eigenvalue  $\lambda_1(\Graph;\mVD)$: in~\autoref{sec:torsional} we show how expanders can be used to study \textit{optimality} of upper bounds: our second main result, \autoref{rem:no-polya-szego-improved}, proves that on metric graphs the Pólya--Szeg\H{o} inequality is asymptotically sharp. This proof is technically subtler than those in \autoref{sec:failure}, since it is based on asymptotic arguments along both parameters of the Ramanujan graphs.
In \autoref{sec:meandistdir} we show how our main theoretical framework, while fairly robust, requires a fine tuning that the currently available expander theory is not delivering.

{
Overall, throughout this article we have \underline{disproved} that the following bounds hold:
\begin{itemize}
\item $\lambda_2(\Graph)\lesssim \frac{1}{|\Graph|^2}$, for any unilateral metric graph $\Graph$, whether non-bipartite or bipartite (\autoref{prop:no-volume-bound}; counterexample: any sequence of Ramanujan graphs -- other than cycles -- with strictly increasing number of vertices);
\item $\lambda_2(\Graph)\lesssim \frac{1}{\Diam(\Graph)^2}$, for any unilateral metric graph $\Graph$, whether non-bipartite or bipartite   (\autoref{rem:no-diam-bound};
counterexample:  any sequence of $d$-regular Ramanujan graphs with strictly increasing number of vertices);
\item $\lambda_2(\Graph)\lesssim \frac{1}{\mGirth(\Graph)^2}$, for any unilateral metric graph $\Graph$, whether non-bipartite or bipartite  (\autoref{rem:no-girth-bound};
counterexample:  either non-bipartite or bipartite LPS expander);
\item $\lambda_2(\Graph)\lesssim \frac{1}{\rho(\Graph)^2}$, for any unilateral metric graph $\Graph$, whether non-bipartite or bipartite  (\autoref{rem:no-meandist-bound};
counterexample:  any sequence of $d$-regular Ramanujan graphs with strictly increasing number of vertices);
\item $\lambda_2(\Graph) \lesssim \frac{F_\alpha(\Diam(\Graph), \rho(\Graph), \mGirth(\Graph))}{|\Graph|^{2+\alpha}}$, for any $\alpha>-2$, any
positive, continuous, $\alpha$-homogeneous function  $F_\alpha$, and any unilateral metric graph $\Graph$, whether non-bipartite or bipartite  (\autoref{cor:lowerb}; counterexample: either non-bipartite or bipartite LPS expander);
\item $\lambda_2(\Graph) \lesssim{F_{-2}(\Diam(\Graph), \rho(\Graph), \mGirth(\Graph))}$, for any positive, continuous, $(-2)$-homogeneous function  $F_{-2}$ and any unilateral metric graph $\Graph$, whether non-bipartite or bipartite
 (\autoref{cor:lowerb}; counterexample: either non-bipartite or bipartite LPS expander);
\item $\lambda_2(\Graph) \gtrsim \frac{|\Graph|^{\alpha}}{F_{2+\alpha}(\Diam(\Graph), \rho(\Graph), \mGirth(\Graph))}$, for any $\alpha>0$, any positive, continuous, $(2+\alpha)$-homogeneous function  $F_{2+\alpha}$, and any unilateral metric graph $\Graph$, whether non-bipartite or bipartite (\autoref{cor:lowerb}; counterexample: either non-bipartite or bipartite LPS expander);
\item $\lambda_1(\Graph;\mVD)T(\Graph;\mVD)\le C |\Graph|$, for any $C<1$ and any non-bipartite unilateral metric graph $\Graph$ (\autoref{rem:no-polya-szego-improved}; counterexample: non-bipartite LPS expander);
\item $\lambda_1(\Graph;\mVD)\lesssim \frac{1}{\Inr(\Graph;\mVD)^2}$, for any unilateral metric graph $\Graph$, whether non-bipartite or bipartite (\autoref{theo:no-inr-bound}; counterexamples: any sequence of non-bipartite Ramanujan graphs with strictly increasing number of vertices; balls exhausting the $d$-regular tree).
\end{itemize}
}

\section{Laplacians on graphs}\label{sec:laplacianseverywhere}

\subsection{Combinatorial graphs}\label{sec:comblapl}
Let $\mG=(\mV,\mE)$ be a finite simple, undirected, unweighted, connected graph with vertex set $\mV$ and edge set $\mE$: to avoid confusion, we refer to it as \textit{combinatorial graph} (or, sometimes, simply as \emph{graph}). If two vertices $\mv,\mw$ are adjacent, then we write $\mv\sim \mw$ (or $\mv\stackrel{\me}{\sim} \mw$ if we want to identify the edge $\me$ that links $\mv,\mw$: in that case, we write $\me=\{\mv,\mw\}$).

We can associate with $\mG$ a \emph{Laplacian}, i.e., a matrix $L_\mG:=\Id_\mV-D^{-1}A$ indexed in $\mV$: here $A$ is the adjacency matrix of $\mG$, i.e.,
\[
A_{\mv\mw}=\begin{cases}
1 \qquad &\hbox{if $\mv,\mw$ are adjacent},\\
0 &\hbox{otherwise}.
\end{cases}
\]
The vector $A\mathbf{1}$ is the \emph{degree sequence} of $\mG$, i.e.: the degree $\deg(\mv):=(A\mathbf{1})_{\mv}$ of $\mv\in \mV$ is the number of vertices $\mw$ that are adjacent to $\mv$ and $D:=\diag(\deg(\mv))_{\mv\in \mV}$. The Laplacian $L_\mG$ is a linear operator on a finite dimensional space, but what is crucial is that its operator norm is uniform in the number of vertices of $\mG$. (This is why it is often referred to as \textit{normalised} Laplacian, to distinguish it from its older sibling, the (unnormalised) Laplacian $\tilde{L}_\mG:=D-A$, which was studied as early as \cite{Kir47}.)

The spectral properties of $L_\mG$ are tightly related to the ``geometric'' properties of $\mG$, like connectedness, symmetries, isoperimetric properties, see \cite{Chu97}.
Observe that while $L_\mG$ is not Hermitian, it is similar (via $D^\frac{1}{2}$) to the real, symmetric matrix $\Id_\mV-D^{-\frac{1}{2}}AD^{-\frac{1}{2}}$. Accordingly, its spectrum is real. An equivalent, more operator theoretical way of framing this issue is that (unlike $\tilde{L}_\mG$) $L_\mG$ is not self-adjoint with respect to  norm
\[
\|f\|_{\ell^2(\mV)}:=\left(\sum_{\mv\in \mV}|f(\mv)|^2\right)^\frac12,\qquad f:\mV\to \C,
\]
but it \textit{is} self-adjoint in the space of square summable functions for the atomic measure with respect to another density, namely
\[
\|f\|_{\ell^2_{\deg}(\mV)}:=\left(\sum_{\mv\in \mV}|f(\mv)|^2 \deg(\mv)\right)^\frac12,\qquad f:\mV\to \C.
\]
Both are Hilbert space norms (indeed, they are equivalent, as long as $\mV$ is a finite set; and, in the regular case of $\deg\equiv d$ we will deal with in this paper, $\|\cdot\|_{\ell^2_{\deg}}=d\|\cdot\|_{\ell^2}$). We will in the following only focus on the metric measure space $(\mG,\dist_\mG,\deg)$, where $\dist_\mG$ is the canonical \textit{shortest path distance} on $\mG$. We refer to \cite{Die17} for this and further elementary notions of graph theory.

It can be checked directly that, for any $f:\mV\to \C$, the function $Lf:\mV\to \C$ is given by
\[
L_\mG f(\mv)=\frac{1}{\deg(\mv)}\sum_{\substack{\mw\in \mV\\ \mw\sim \mv}} (f(\mv)-f(\mw)),\qquad \mv\in \mV.
\]
A direct computation shows that $L_\mG$ is the self-adjoint operator on $\ell^2_{\deg}(\mV)$ associated with the quadratic form $\mathfrak{a}$ given by
\[
\mathfrak{a}(f):=\sum_{\me\in \mE}|\nabla f(\me)|^2,\qquad f:\mV\to \C,
\]
where $\nabla f$ is the discrete gradient defined by
\[
\nabla f(\me):=f(\mv)-f(\mw)\qquad \hbox{whenever }\me=\{\mv,\mw\}.
\]

Because $L_\mG$ is positive semi-definite in the Hilbert space $\ell^2_{\deg}(\mV)$, all its $\#\mV$ eigenvalues are non-negative.
The lowest eigenvalue is always 0: it is simple, with the corresponding eigenspace  spanned by the constant functions on $\mV$.
It follows from Ger\v{s}gorin's Theorem that all eigenvalues  of $L_\mG$ lie in the interval $[0,2]$:
\[
0=\nu_1(\mG)\le\nu_2(\mG)\le\ldots \nu_{\#\mV}(\mG)\le 2;
\]
and, in fact, $\nu_2(\mG)>0$ since $\mG$ is assumed to be connected.
By Courant's minmax principle, these eigenvalues can be determined in terms of the Rayleigh quotient
\[
R(\mG,f):=\frac{\mathfrak{a}(f)}{\|f\|^2_{\ell^2_{\deg}(\mV)}},\qquad f:\mV\to \C,
\]
so in particular
\begin{equation}\label{eq:courantf-disc}
\nu_2(\mG)=\min_{\substack{f:\mV\to\C \\ \sum\limits_{\mv\in\mV}f(\mv)\deg(\mv)=0}} \frac{\mathfrak{a}(f)}{\|f\|^2_{\ell^2_{\deg}(\mV)}}.
\end{equation}



\subsection{Foundations of Expander Theory}\label{sec:foundations}

Let $\mG = (\mV,\mE)$ be as in \autoref{sec:comblapl}, and let us furthermore assume  $\mG$ to $d$-regular, i.e., $\deg(\mv)=d$ for all $\mv\in\mV$. Here and in the following, we let for all $A\subset \mV$, 
\[
\partial A := \{ \{\mv,\mw\} \in \mE : \mv \in A, \mw \notin A \}\quad \hbox{and}\quad \Vol(A):=\sum_{\mv\in A}\deg(\mv)=d\#A.
\]
Also, the \emph{Cheeger constant} 
 $h(\mG)$ of $\mG$ is defined as
\begin{equation}\label{eq:cheeger-two-sided}
h(\mG) = \min_{\emptyset \ne A \subsetneq \mV} \frac{\#\partial A}{\min\{\Vol(A),\Vol(\mV\setminus 	A)\}}.
\end{equation}
Because each $\mv\in A$ has at most $d=\deg(\mv)$ neighbours outside $A$, 
 $\#\partial A\le \Vol(A)$; and, likewise, $\#\partial A\le \Vol(\mV\setminus A)$. So, $h(\mG)\le 1$ and $\nu_2(\mG)\le 2$. What about lower bounds?

It was shown in \cite{Fie73} that the size of $\nu_2$ is a quantitative measure of the connectedness of $\mG$: the larger $\nu_2$, the more connected $\mG$. In the same years, the theory of families of graphs display strong connectedness was being explored \cite{Mar73}. This eventually led to the theory of expanders.

\begin{defi}
\label{def:combin-exp}
A sequence of $d$-regular graphs $(\mG_n)_{n \in \N}$ with $d$ independent of $n$ and $\#\mV_n \to \infty$ is a \emph{combinatorial expander family} if there exists $\epsilon > 0$ such that 
\[
h(\mG_n) \ge \epsilon\qquad\hbox{ for all }n\in \N.
\]
\end{defi}

\begin{defi}
A sequence of $d$-regular graphs $(\mG_n)_{n \in \N}$ is a \emph{spectral expander family} if there exists $\tilde\epsilon > 0$ 
such that 
\begin{equation}
\nu_2(\mG_n) \ge \tilde\epsilon\qquad\hbox{for all }n\in \N.
\end{equation}
\end{defi}

The combinatorial and spectral formulations are linked fundamentally by the \emph{Cheeger Inequality}
\begin{equation}
\label{eq:cheeger}
\frac{h^2(\mG)}{2}< \nu_2(\mG)\le  2h(\mG)
\end{equation}
cf.~\cite{AloMil84} and~\cite[Section~2.3]{Chu97}: so, the lowest positive eigenvalue of a connected graph $\mG$ cannot be made arbitrarily small. If, additionally, $\mG$ is regular, then the \emph{Alon--Boppana bound} in \cite[Theorem~1]{Nil91} dictates that for every graph $\mG$ that contains two edges of distance $2\delta+2$ there holds
\begin{equation}\label{eq:alon-boppana}
\nu_2(\mG)\le 1-\frac{2\sqrt{d-1}}{d}+\frac{2\sqrt{d-1}-1}{d(\delta+1)}
\end{equation}
and, thus, 
\[
\liminf_{\#\mV\to\infty} (1-\nu_2(\mG))\ge \frac{2\sqrt{d-1}}{d}.
\]

Graphs that saturate the asymptotic Alon--Boppana bound and, thus, satisfy an optimal version of~\eqref{eq:cheeger-two-sided}, are especially interesting. We assume throughout that $\mG$ has at least three vertices, to avoid trivialities.

\begin{defi}\label{def:raman}
A connected $d$-regular graph $\mG$ is a \emph{Ramanujan graph} if it is bipartite and
\begin{equation}\label{eg:def-raman-bip}
|1-\nu_2(\mG)|\le  \frac{2\sqrt{d-1}}{d};
\end{equation}
or else, if it is non-bipartite and
\begin{equation}\label{eg:def-raman-non-bip}
\max\{|1-\nu_2(\mG)|,|1-\nu_{\#\mV}(\mG)|\} \le \frac{2\sqrt{d-1}}{d}.
\end{equation}
\end{defi}
In particular, infinite families of Ramanujan graphs with growing number of vertices are examples of expanders.

\begin{rem}\label{rem:eigenv-triv}
The definition of Ramanujan graph distinguishes both cases because $\nu_{\#\mV}(\mG)$ is well-known to be always $2$ if $\mG$ is bipartite, which would trivially contradict the defining bound. The eigenvalues $0,2$ are, indeed, regarded as trivial.
\end{rem}

Ramanujan graphs do exist: indeed, cycle graphs on any number of vertices are Ramanujan. So, can the Ramanujan property hold for regular graphs of arbitrarily high degree?
Yes, it can, as shown by the (non-bipartite) \textit{complete graph} $\mK_{\#\mV}$, for which $\nu_{2}=\nu_{\#\mV}=\frac{\#\mV}{\#\mV-1}$;
and also for the (bipartite) \textit{complete biregular graph} $\mK_{\#\mV,\#\mV}$, for which $\nu_{2}=\nu_{\#\mV-1}=1$: so that, in either case, the condition in \autoref{def:raman} is satisfied for any $\#\mV$. However, for both complete and complete biregular graphs the degree of regularity grows linearly with the number of vertices, which is usually very restrictive.

So, the real breakthrough came when, Alexander Lubotzky, Ralph Phillips and Peter Sarnak proved in \cite[Theorem~4.1]{LubPhiSar88} that for infinitely many pairs $(p,q)$ of primes there is a Ramanujan graph $\mG^{p,q}$ with degree 
\[
d=p+1:
\] they are usually referred to as \emph{LPS expanders}. They were constructed in~\cite[Section~2]{LubPhiSar88} and, independently, in \cite{Mar88}:
\begin{itemize}
\item If $p$ is a quadratic residue modulo $q$, $\left(\frac{p}{q}\right)=1$, then $\mG^{p,q}$ is a Cayley graph over the projective special linear group $\text{PSL}_2(\mathbb{F}_q)$: it is \underline{non-bipartite}, $d$-regular and has 
\[
\#\mV=\frac{q(q^2-1)}{2}
\]
vertices.
\item If  $p$ is a 	non-residue modulo $q$, $\left(\frac{p}{q}\right)=-1$, the graph $\mG^{p,q}$ is constructed over the full projective general linear group $\text{PGL}_2(\mathbb{F}_q)$: it is \underline{bipartite}, $d$-regular and has 
\[
\#\mV={q(q^2-1)}
\]
vertices.
\end{itemize}
For all our purposes in this article, LPS expanders will suffice. Though, let us also mention that, more recently, Adam Marcus, Daniel Spielman and Nikhil Srivastava provided a more flexible construction, proving in~\cite[Theorem~5.5]{MarSpiSri15} and \cite[Theorem~1.1]{MarSpiSri18} the existence of a bipartite Ramanujan graphs of \textit{any} degree $d$ on an arbitrarily large, even number of vertices $\#\mV\ge 2d$: 
\begin{itemize}
\item Starting from any complete bipartite graph $\mK_{d,d}$, which is a Ramanujan graph, the existence of a new bipartite, $d$-regular graph on $2^{k+1}d$ vertices -- a so-called \textit{2-lift} of $\mK_{d,d}$ -- is proven all of whose non-trivial eigenvalues are distant from 1 at most $\frac{2\sqrt{d-1}}{d}$.
\item Alternatively, a model where a $d$-regular bipartite graph (indeed, possibly a multigraph) on any even number of vertices is generated by taking the union of $d$ random perfect matchings across a fixed bipartition exists.
\end{itemize}
We refer to them both as \textit{MSS expanders}. 
In particular, Ramanujan graphs of even degree -- hence, Eulerian -- exist. 

\begin{rem}
The notion of Ramanujan graphs can be extended to biregular graphs and, indeed, \cite[Theorem~5.6]{MarSpiSri15} states the existence of an infinite sequence of $(c, d)$-biregular bipartite Ramanujan graphs for all $c, d \ge 3$; but this seems to be not directly relevant to the purpose of the present article.

Further classes of expanders are surveyed in \cite[Section~6.3]{Chu97} and \cite[Chapter~4]{Kow19}.
\end{rem}

\subsection{Metric graphs}
The \emph{metric graph} $\Graph$ associated with a combinatorial graph $\mG$ is obtained by associating each edge $\me$ of a combinatorial graph $\mG$ with an interval $(0,\ell_\me)$ (see \cite{Mug19} for a more precise, abstract definition), with $\ell_\me \in (0,\infty)$; we call $\Graph$ \emph{unilateral} if $\ell_\me\equiv 1$. The metric graph $\Graph$ is, thus, a metric measure space $(\Graph,\dist_{\Graph},\lambda)$, with respect to the shortest path metric $d$ induced by the edgewise Euclidean distance; and the sum $\lambda$ of edgewise Lebesgue measures.

This defines function spaces $C(\Graph)$ (with respect to the shortest path metric based on the edgewise Euclidean distance) and $L^2(\Graph)$ (with respect to the sum measure) and, therefore, also the Sobolev space
\[
H^1(\Graph):=\{f\in C(\Graph)\cap L^2(\Graph):f'\in L^2(\Graph)\},
\]
which is a Hilbert space for the canonical inner product
\[
(f,g)_{H^1}:=\int_\Graph \left(f'(x)\overline{g'(x)}+ f(x)\overline{g(x)} \right) \dd x,\qquad f,g\in H^1(\Graph),
\]
and its closed subspace
\[
H^1_0(\Graph;\mVD):=\{f\in C(\Graph)\cap L^2(\Graph):f'\in L^2(\Graph)\hbox{ and }f(\mv)=0\hbox{ for all }\mv\in \mVD\},
\]
where $\mVD$ is some given subset $\mVD$ of $\mV$: we refer to \cite{Mug19} for further details. Of course, $H^1_0(\Graph;\mVD)=H^1(\Graph)$ if (and only if) $\mVD=\emptyset$.

The Sobolev seminorm 
\[
a:u\mapsto \int_\Graph |u'|^2\dd x,\qquad u\in H^1_0(\Graph;\mVD),
\]
defines a closed quadratic form with respect to the Hilbert space $H:=L^2(\Graph)$ for any $\mVD$: its associated self-adjoint operator on $L^2(\Graph)$ is  the Laplacian $\DeltaGVD $ on $\Graph$: its domain consists of all functions of class $L^2(\Graph)$ that are edgewise twice weakly differentiable with second weak derivative of class $L^2(\Graph)$, and that satisfy Dirichlet conditions at the vertices in $\mVD$, and standard (i.e., continuity and Kirchhoff) conditions at all other vertices. We use the simpler notation $\Delta_{\Graph}$ if $\mVD=\emptyset$.

By~\cite[Théorème~1.2]{Nic87b}, the metric graph Laplacian  $\DeltaGVD $ is an unbounded operator that has compact (indeed, trace class) resolvent and, hence, pure point spectrum, which we denote by
\[
\lambda_1(\Graph;\mVD)\le \lambda_2(\Graph;\mVD)\le \ldots\le \lambda_n(\Graph;\mVD)\le \ldots\nearrow +\infty.
\]
For the sake of notational simplicity, in the important special case $\mVD=\emptyset$, we write $\lambda_k(\Graph)$ instead of $\lambda_k(\Graph;\mVD)$. Because of the Poincaré inequality for metric graphs, $\lambda_1(\Graph;\mVD)=0$ if and only if $\mVD=\emptyset$, and in this case the corresponding eigenspace  is spanned by the constant functions on $\Graph$. We refer to the smallest strictly positive eigenvalue -- i.e., $\lambda_2(\Graph)$ if $\mVD=\emptyset$ or $\lambda_1(\Graph;\mVD)$ if $\mVD\ne\emptyset$ -- as the \emph{spectral gap} (of $\Graph$ with respect to $\mVD$).

By Courant's minmax principle, all these eigenvalues are associated with the Rayleigh quotient
\[
R(\Graph,u):=\frac{{a}(u)}{\|u\|^2_{L^2(\Graph)}},\qquad u\in H^1_0(\Graph;\mVD),
\]
and in particular there holds
\begin{equation}\label{eq:courantf-cont}
\lambda_2(\Graph)=\min_{\substack{u\in H^1(\Graph) \\ \int_\Graph u\dd x=0}} \frac{{a}(u)}{\|u\|^2_{L^2(\Graph)}}\qquad\hbox{and}\qquad 
\lambda_1(\Graph;\mVD)=\min_{u\in H^1_0(\Graph;\mVD)} \frac{{a}(u)}{\|u\|^2_{L^2(\Graph)}}.
\end{equation}

It follows from \cite[Theorem~4.2]{KenKurMal16} that each unilateral metric graph $\Graph$ satisfies $\lambda_2(\Graph)\le \pi^2$, unless $\Graph$ is a loop (in which case $\lambda_2(\Graph)= 4\pi^2$). Can this upper bound on unilateral metric graphs be improved? It is natural to wonder whether these eigenvalues admit bounds based on the geometry of the metric graph $\Graph$; the most natural such quantities are based on the measure or metric theoretical structure of $\Graph$, including its \emph{volume}
\[
|\Graph|:=\sum_{\me\in\mE}\ell_\me
\]
and its \emph{diameter}
\[
\Diam(\Graph):=\max_{x,y\in \Graph}\dist_\Graph(x,y).
\]

Early fundamental contributions to estimates on $\lambda_k(\Graph)$ were developed in \cite{Nic87,Fri05}, where it was shown that lower bounds in terms of $|\Graph|$ exist.
A more comprehensive study of lower \textit{and} upper bounds on $\lambda_2(\Graph)$ was initiated in~\cite{KenKurMal16}. In particular, both a lower bound and an upper bound on $\lambda_2(\Graph)$ is available in terms of combinations of $|\Graph|$ \textit{and} $\Diam(\Graph)$, \cite[Theorem 7.2]{KenKurMal16} and \cite[Theorem 1.1]{Ken20}.

However, it was observed in \cite[Section 4.2 and Theorem~5.10]{KenKurMal16} that for general (i.e., not necessarily unilateral) metric graphs, an upper bound on $\lambda_2(\Graph)$ only based on either the volume $|\Graph|$ or the diameter $\Diam(\Graph)$ cannot hold: in other words, there exist a sequence of metric graphs $(\Graph_n)_{n\in\N}$ with fixed diameter  such that $\lambda_2(\Graph_n)\to \infty$; and a sequence of metric graphs $(\Graph'_n)_{n\in\N}$ with fixed volume such that $\lambda_2(\Graph'_n)\to \infty$.

\section{Failure of upper bounds}\label{sec:failure}

The eigenvalues of the (normalised) graph Laplacian $L_\mG$ and the metric graph Laplacian $\Delta_{\mG;\mVD}$ are tightly related: The following crucial characterisation is an immediate consequence of a transference principle that was obtained by Joachim von Below in \cite[Theorem, page 320]{Bel85}.

\begin{theo}\label{theo-vonbelow}
Let $\mG$ be a (finite simple, undirected, unweighted, connected) graph, and let $\Graph$ be the associated unilateral metric graph. 

Then for all $0\le \lambda$ such that $\lambda\ne \pi^2k^2$ for any $k\in \N$, $\lambda$ is an eigenvalue of $\Delta_\Graph$ if and only if $1-\cos\sqrt{\lambda}$ is an eigenvalue of $L_\mG$, and in this case the eigenvector of $L_\mG$ is the vector of the evaluation of the eigenfunction of $\Delta_\Graph$ at the vertices.

In particular, if $\Graph$ is not a loop,
\begin{equation}\label{eq:vonbelow-ping}
\lambda_2(\Graph)=\arccos(1-\nu_2(\mG))^2.
\end{equation}
\end{theo}
(Strictly speaking, we identify in~\eqref{eq:vonbelow-ping} the principal branch as $\lambda_2(\Graph)$. This is valid in the present setting  because of the known bound $\lambda_2(\Graph)\le \pi^2$, apart from the case of the loop.)

Combining the definition of Ramanujan graph and the Alon--Boppana bound \eqref{eq:alon-boppana} we find
\[
0\le \frac{2\sqrt{d-1}}{d}-\frac{2\sqrt{d-1}-1}{d(\delta+1)}\le 1-\nu_2(\mG)\le \frac{2\sqrt{d-1}}{d}:
\]
Since $\arccos$ is strictly monotonically decreasing, 
 we immediately obtain the following.

\begin{cor}\label{rem:quantumraman}
Given any Ramanujan graph $\mG$ of degree $d$, the associated unilateral metric graph $\Graph$ satisfies
\begin{equation}\label{eq:low-bound-lambda2}
\arccos\left(\frac{2\sqrt{d-1}}{d}\right)^2\le \lambda_2(\Graph).
\end{equation}
If, additionally,  $\nu_2(\mG)\le 1$, then
\begin{equation}\label{eq:low-bound-lambda2}
\lambda_2(\Graph)\le \arccos (0)^2.
\end{equation}
\end{cor}

Recall that LPS expanders are family of Ramanujan graphs $\mG^{p,q}$ of degree $d=p+1$ that display a small-world behaviour, with logarithmically growing diameter and girth: this carries over to their unilateral metric versions, too, which we call \emph{Ramanujan metric graphs} and denote by $\Graph^{p,q}$. We likewise refer as \emph{LPS metric expanders} to the sequence of metric graphs whose underlying combinatorial graphs are LPS expanders.
In particular, no upper bound of $\lambda_2(\Graph)$ by a quantity that converges to 0 along a sequence of Ramanujan metric  graphs of fixed degree and increasing vertex number can hold.

\subsection{Volume}

As a warm up, let us sharpen an observation in \cite[Section~4.2]{KenKurMal16}.
The following holds.

\begin{prop}\label{prop:no-volume-bound}
There is \underline{no} universal constant $C>0$ such that
\begin{equation}\label{eq:no-upper-bound-v}
\lambda_2(\Graph)\le \frac{C}{|\Graph|^{2}}
\end{equation}
for all -- either bipartite or non-bipartite --  unilateral metric graphs $\Graph$.
\end{prop}

That is, the upper bound in~\eqref{eq:no-upper-bound-v} generally fails even under the condition of unilaterality.
 
\begin{proof}
Consider any sequence $(\mG_n)$ of Ramanujan graphs of degree $d_n\ge 3$ on a strictly increasing number of vertices. Each $\mG_n$ has $\#\mV_n$ vertices and hence, by the Handshaking Lemma, 
$\#\mE_n=\frac{d_n \#\mV_n}{2}\nearrow \infty$ edges. Clearly, this is also the volume $|\Graph_n|$ of the unilateral metric graph $\Graph_n$ associated with $\mG_n$. 
Now, the right-hand side of \eqref{eq:no-upper-bound-v} converges to $0$ as $n$, hence $\#\mV$ tends to $\infty$, whereas the left-hand side remains bounded from below in view of \eqref{eq:low-bound-lambda2}.
\end{proof} 
 
Let us stress that this proof does not rely on the theory of expanders: for example, the sequence $(\mK_n)$ of complete graphs on $n$ vertices suffices to perform the proof: this can also be seen looking at \cite[formula (3.5)]{KenKurMal16}.
 
\subsection{Diameter}
Another classical quantity attached to a graph $\mG$ is the \emph{diameter}, defined as
\[
\Diam(\mG):=\max_{\mv,\mw}\dist_\mG (\mv,\mw).
\]
For general $d$-regular graphs, $d\ge 3$, the diameter can be bounded by
\begin{equation}\label{eq:bound-general}
\log_{d-1}(\#\mV) -\mathcal O(1)\le \Diam(\mG) \le \mathcal O(\#\mV) \qquad \hbox{as }\#\mV\to \infty;
\end{equation}
the lower bound is the classical \emph{Moore bound}, whereas the upper bound is sharp for cycles ($d=2$), cube-connected cycles ($d=3$) and, more generally, cycle of expanders.

\begin{rem}
Expanders satisfy sharper upper bounds: for any $d$,
\begin{equation}\label{eq:bound-diam-chung-sardari}
\frac{4}{3}\log_{d-1}(\#\mV) - {\mathcal O}(1) \le \Diam(\mG^{p,q}) \le  2\log_{d-1}(\#\mV) + {\mathcal O}(1)\qquad\hbox{as }\#\mV\to\infty
\end{equation}
hold for the non-bipartite LPS expanders: the upper bound was remarked in~\cite[Section~1.1]{LubPer16} based on \cite[Corollary~5.5]{ChuFabMan94} (in fact, it holds even for general Ramanujan graphs), whereas the lower bound was proved in~\cite[Theorem~1.2]{Sar19}.
\end{rem}

The following sharpens \cite[Theorem~5.10]{KenKurMal16}, where the existence of a sequence of \textit{possibly non-unilateral} metric graphs producing a counterexample was obtained with a rather involved proof. 

\begin{prop}\label{rem:no-diam-bound}
There is \underline{no} universal constant $C>0$ such that
\[
\lambda_2(\Graph)\le \frac{C}{\Diam(\Graph)^{2}}
\]
for all -- either bipartite or non-bipartite -- unilateral metric graphs $\Graph$.
\end{prop}

\begin{proof}
Let $d$ be fixed, $d\ge 3$, and let $(\mG_n)$ be a sequence of $d$-regular Ramanujan graphs on a strictly increasing number of vertices. It now suffices to observe that the diameter of any unilateral metric graph $\Graph$ can be trivially bounded from below in terms of the diameter of the associated combinatorial graph. 
In view of \eqref{eq:low-bound-lambda2} and the lower bound in~\eqref{eq:bound-general}, this proves the claim.
\end{proof}

\begin{rem}
Comparing \eqref{eq:bound-general} with \eqref{eq:bound-diam-chung-sardari} shows that LPS expanders do \emph{not} necessarily satisfy much better bounds than general graphs. Cycle graphs, for example, have much higher diameter ($\Diam(\mC_{2n})=\mathcal O(n)$, of course); however, the spectral gap $\lambda_2(\Graph)$ of the corresponding metric graph tends to $0$ as precisely the same rate as $\Diam(\Graph)^2$, and so they are useless for the purpose of proving \autoref{rem:no-diam-bound}.

This explains the appeal of expanders: their crucial feature, see \autoref{rem:quantumraman}, is that they guarantee the existence of an \textit{infinite} sequence of growing graphs that share the same lower bound on the spectral gap, independently of the graph's size, while crucial metric quantities can simultaneously blow up.
\end{rem}

\begin{rem}
The asymptotic bounds in~\eqref{eq:bound-diam-chung-sardari} extend to the \textit{radius} $\rad$ of the non-bipartite LPS expanders, too, due to the classical relation $\rad\le\Diam\le 2\rad$ that holds on any metric space.
\end{rem}

\subsection{Girth}
Given a combinatorial graph $\mG$, and the associated unilateral metric graph $\Graph$, the \emph{girth} $\mGirth(\mG)$ (resp., $\mGirth(\Graph)$) is defined as the length of the shortest cycle in $\mG$  (resp., $\Graph$); by convention, the girth of a tree is 0. LPS expanders have logarithmically growing girth: more precisely,  by~\cite[Remark 4.3.4 and Corollary~4.3.6]{DavSarVal03} for $q\gg p$
\begin{equation}\label{eq:bigbos2}
2+2 \log_{d-1}(\#\mV)\ge \mGirth(\mG^{p,q})\ge 
\begin{cases}
\frac{2}{3}\log_{d-1}(\#\mV)\quad &\hbox{in the non-bipartite case},\\
\frac{4}{3}\log_{d-1}(\#\mV)-\log_{d-1} 4 &\hbox{in the bipartite case},
\end{cases}
\end{equation}
and even
\begin{equation}\label{eq:bigbos2-v}
\mGirth(\mG^{p,q}) \asymp \frac{4}{3}\log_{d-1}(\#\mV)\to \infty \qquad \hbox{as }\#\mV\to \infty
\end{equation}
for bipartite LPS expanders (cf.~\cite[Corollary]{BigBos90}).
Observing that the girth of $\mG$ agrees with the girth of $\Graph$, and reasoning precisely as in the proof of \autoref{rem:no-diam-bound}, the following can be proven: this was observed already in \cite[Example~1.6]{BerKenKur23}.

\begin{prop}\label{rem:no-girth-bound}
There is \underline{no} universal constant $C>0$ such that
\[
\lambda_2(\Graph)\le \frac{C}{\mGirth(\Graph)^{2}}
\]
for all -- either bipartite or non-bipartite --  unilateral metric graphs $\Graph$ with $\mGirth(\Graph)>0$.
\end{prop}

The logically weaker assertion that $\lambda_2(\Graph)$ does not admit a uniform upper bound by the length of the \textit{longest} cycle in $\Graph$ had already been proved in~\cite[Section~5.3]{KenKurMal16}.

\subsection{Mean Distance}
The mean distance
\begin{equation}
\rho(\mG) := \frac{2}{\#\mV(\#\mV-1)} \sum_{\substack{\mv,\mw\in\mV\\ \mv < \mw}} \dist_\mG (\mv,\mw)
\end{equation}
of a combinatorial graph $\mG$ represents the expected shortest-path distance between a uniformly selected pair of distinct vertices (here $\mv<\mw$ means $i<j$ for any enumeration $\mV=\{ \mv_1,\mv_2,\ldots\}$ such that $\mv_i=\mv$ and $\mw=\mv_j$).
It has been elaborated on since~\cite{DoyGra77}, while its interplay with spectral theory has been studied at least since~\cite{Moh91b}. Its natural counterpart 
\begin{equation}
\rho(\Graph) := \frac{1}{|\Graph|^2} \int_\Graph\int_\Graph \dist_\Graph (x,y)\dd x \dd y
\end{equation}
 for metric graphs has been studied in~\cite{GarMarSil23,BapKenMug24}. 
 
  We already know from \autoref{rem:no-diam-bound} that no uniform upper bound on $\lambda_2(\Graph)$ by $\Diam(\Graph)^2$ holds. It is natural to wonder if a (weaker) upper bound on $\lambda_2(\Graph)$ by $\rho(\Graph)^2$, instead, can possibly hold: this is~\cite[Open Problem~3.4]{BapKenMug24}. This was disproved by ChatGPT 5.5 pro, as prompted by Robert Joseph George \cite{Geo26}.
 
\begin{prop}\label{rem:no-meandist-bound}
There is \underline{no} universal constant $C>0$ such that
\[
\lambda_2(\Graph)\le \frac{C}{\rho(\Graph)^{2}}
\]
for all -- either bipartite or non-bipartite --  unilateral metric graphs.
\end{prop}

We emphasise that \autoref{rem:no-meandist-bound} logically implies \autoref{rem:no-diam-bound}, since $\rho(\Graph)<\Diam(\Graph)$, by~\cite[Theorem~3.1]{BapKenMug24}

\begin{proof}
We follow the usual scheme: we let $d$ be fixed, $d\ge 3$, and let $(\mG_n)$ be a sequence of $d$-regular Ramanujan graphs on a strictly increasing number of vertices.

We then invoke the lower bound
\begin{equation}\label{eq:meandist-lower}
\rho(\mG)\ge  \log_{d-1}(\#\mV)-\mathcal O(1)\to \infty\qquad \hbox{as }\#\mV\to \infty
\end{equation}
that holds for any $d$-regular graph $\mG$ with $d\ge 3$, see \cite[Formula (5)]{Shi20}. In view of \eqref{eq:low-bound-lambda2}, it suffices to show that the mean distance of any unilateral metric graph $\Graph$ can be bounded from below in terms of the mean distance of the associated combinatorial graph, and then take a sequence of Ramanujan graphs. The only missing step is to show a lower bound of $\rho(\Graph)$ by $\rho(\mG)$:  this is a consequence of \autoref{lem:rho-g-gg} below. The claim is thus proved.
\end{proof}

While it is known that 
\[
\rho(\Graph)\ge \rho(\mG)
\]
is generally false for unilateral metric graphs and their underlying combinatorial graphs, see \cite[Remark~2.2]{GarMarSil23}, a lower estimate
\[
\rho(\Graph)\ge \rho(\mG)-\mathcal{O}(1)
\]
holds: here comes the missing step in \autoref{rem:no-meandist-bound} suggested by ChatGPT.

\begin{lemma}\label{lem:rho-g-gg}
Let $\mG$ be $d$-regular and let
$\Graph$ be the corresponding unilateral metric graph. 
Then
\[
\left|\rho(\Graph)- \rho(\mG)\right| \le  3.
\]
\end{lemma}

\begin{proof}
Choose $X,Y$ independently and uniformly at random with respect to the measure on $\Graph$; by definition of $\rho$,
\begin{equation}\label{eq:step0}
\mathbb E\bigl[\dist_{\Graph}(X,Y)\bigr] = \rho(\Graph).
\end{equation}

Let $\me_X$ and $\me_Y$ be the  edges of $\Graph$ containing $X$
and $Y$, respectively.
Choose $U$ to be one
of the two endpoints of $\me_X$, each with probability $\frac12 $, and likewise
choose $V$ to be one of the two endpoints of $\me_Y$, each with probability
$\frac12 $.

The edge $\me_X$ containing $X\in \Graph$ is any given edge $\me$ with probability $\frac{1}{\#\mE}$. Conditioned on $X\in \me$, the endpoint $U$ equals each endpoint of $\me$ with probability $\frac12 $, so for a fixed vertex $\mv\in \mV$,
\[
\mathbb P(U=\mv) = \sum_{\me\ni \mv} \frac{1}{2\#\mE}= \frac{\deg(\mv)}{2\#\mE} = \frac{d}{2\#\mE}  = \frac{1}{\#\mV},
\]
by the Handshaking Lemma.
Thus $U$ is uniformly
distributed on $\mV$, and by the same argument so is $V$; since $X$, $Y$,
and the two independent endpoint choices are mutually independent, $U$ and
$V$ are independent.

Because $\Graph$ is unilateral, any point of an edge is
within distance $1$ of either endpoint of that edge; in particular
$\dist_\Graph(X,U)\le 1$ and $\dist_\Graph(Y,V)\le 1$.
Moreover, since $U,V\in \mV$, their distance in $\mG$ agrees with their distance in $\Graph$.

The triangle inequality on $\Graph$ therefore gives, for every realization
of $X,Y,U,V$,
\begin{equation}\label{eq:chat-step2}
|\dist_\Graph(U,V) -\dist_\Graph(X,Y)|
\le \dist_\Graph(U,X)  + \dist_\Graph(Y,V)
\le 1 +  1 .
\end{equation}

Since $U,V$ are independent and uniform on $\mV$, taking expectations in \eqref{eq:chat-step2} we see that
\[
\mathbb E\bigl[\dist_\Graph(U,V)\bigr]
= \frac{1}{\#\mV^2}\sum_{\mathsf{u},\mv\in \mV} \dist_\Graph(\mathsf{u},\mv) = \frac{\#\mV-1}{\#\mV}\rho(\mG).
\]
 and using \eqref{eq:step0}, as well as the triangle inequality,
\[
\left| \frac{\#\mV-1}{\#\mV}\rho(\mG) - \rho(\Graph) \right|	\le 2,
\]
which finally yields the claim, since $\rho(\mG)\le \Diam(\mG)< \#\mV$.
\end{proof}

\subsection{Combination of quantities and lower bounds}

We can obtain the following asymptotic behavior for the lowest positive eigenvalue of metric expanders.

\begin{lemma}\label{rem:quantumraman-diam}
Let $d\in \N$ be fixed, $d\ge 3$. Then the non-bipartite LPS metric expanders $(\Graph^{p,q})$ of degree $d\equiv p+1$ satisfy
\begin{equation}\label{eq:upplow-bound-lambda2-ultimate}
 \lambda_2(\Graph^{p,q})\to  \arccos\left(\frac{2\sqrt{d-1}}{d}\right)^2\qquad \hbox{as }\#\mV\to\infty.
\end{equation}
\end{lemma}

\begin{proof}
Given any Ramanujan graph $\mG$ of degree $d$ and diameter $\Diam(\mG)$ large enough, the associated unilateral metric graph $\Graph$ satisfies
\begin{equation}\label{eq:upplow-bound-lambda2}
\arccos\left(\frac{2\sqrt{d-1}}{d}\right)^2\le \lambda_2(\Graph)\le \arccos\left(\frac{2\sqrt{d-1}}{d}-\frac{2\sqrt{d-1}-1}{d\lfloor \frac{\Diam(\mG)}{2}\rfloor}\right)^2,
\end{equation}
where the lower bound is \eqref{eq:low-bound-lambda2} and the upper bound follows from the Alon--Boppana bound \eqref{eq:alon-boppana} and the monotonicity of $\arccos$. The claim now follows 
combining \eqref{eq:bound-diam-chung-sardari} (which holds for non-bipartite LPS expanders) with \eqref{eq:upplow-bound-lambda2} .
\end{proof}

This can be used to extend our previously obtained result to fractions of metric quantities. 
In fact, \autoref{rem:quantumraman-diam} suggests that expanders may be used to prove the failure of lower and upper bounds alike. 
Observe that by elementary scaling arguments, upper or lower bounds may only be achievable by quantities that are $(-2)$-homogeneous.
For this reason, we are going to consider quantities of the form
\[
\frac{\Phi_1(\Graph)^s}{\Phi_2(\Graph)^{2+s}},\qquad s\in \mathbb R,
\]
where $\Phi_1,\Phi_2$ are any of the quantities we have considered in the previous sections. If $s=0$, then we know already that there is no universal constant $C>0$ such that
\[
\lambda_2(\Graph)\le C\frac{1}{\Phi_2(\Graph)^{2}}
\]
for all unilateral metric graphs if $\Phi_2(\Graph)=|\Graph|$ (\autoref{prop:no-volume-bound}), $\Phi_2(\Graph)=\Diam(\Graph)$ (\autoref{rem:no-diam-bound}), 
$\Phi_2(\Graph)=\mGirth(\Graph)$ (\autoref{rem:no-girth-bound}), $\Phi_2(\Graph) =\rho(\Graph)$ (\autoref{rem:no-meandist-bound}).

We can state the following, which complements some results in~\cite[Section~4.2]{BapKenKra26}.

\begin{theo}\label{cor:lowerb}
The following assertions hold.
\begin{enumerate}[(1)]
\item For any $\alpha>-2$ and any positive, continuous, $\alpha$-homogeneous function $F_{\alpha}:(0,\infty)^3\to \R$ there is \underline{no} universal constant $C > 0$ such that 
\[
\lambda_2(\Graph) \le C\; \frac{F_\alpha(\Diam(\Graph), \rho(\Graph), \mGirth(\Graph))}{|\Graph|^{2+\alpha}}
\]
for all  -- either bipartite or non-bipartite -- unilateral metric graphs with $\mGirth(\Graph)>0$.
\item For any positive, continuous,  $(-2)$-homogeneous function $F_{-2}:(0,\infty)^3\to \R$ there is \underline{no} universal constant $C > 0$ such that 
\[
\lambda_2(\Graph) \le C\; {F_{-2}(\Diam(\Graph), \rho(\Graph), \mGirth(\Graph))}
\]
for all -- either bipartite or non-bipartite --  unilateral metric graphs with $\mGirth(\Graph)>0$.
\item For any $\alpha>0$ and any  positive, continuous, $(2+\alpha)$-homogeneous function $F_{2+\alpha}:(0,\infty)^3\to \R$ there is \underline{no} universal constant $c > 0$ such that 
\[
\lambda_2(\Graph) \ge c\; \frac{|\Graph|^{\alpha}}{F_{2+\alpha}(\Diam(\Graph), \rho(\Graph), \mGirth(\Graph))}
\]
for all -- either bipartite or non-bipartite --  unilateral metric graphs with $\mGirth(\Graph)>0$.
\end{enumerate}

%
\end{theo}

\begin{proof}
For any $d$-regular graph $\mG$ on $\#\mV$ vertices and the corresponding unilateral metric graph $\Graph$ we have 
\[
|\Graph|=\frac{d\#\mV}{2}.
\]
On the other hand, we have seen that, for fixed $d=p+1$, all $d$-regular Ramanujan metric graphs $\Graph^{p,q}$ satisfy
\[
\Diam(\Graph^{p,q}), \rho(\Graph^{p,q}) \asymp \log_{d-1}(\#\mV)\qquad \hbox{as }\#\mV\to\infty;
\]
and, likewise, LPS expanders satisfy
\[
\mGirth(\Graph^{p,q}) \asymp \log_{d-1}(\#\mV)\qquad \hbox{as }\#\mV\to\infty.
\]

Also, we know that for all $d$-regular graphs $\mG$ and the associated unilateral metric graph, and up to a correction term that does not depend on $\mG,\Graph$,
\[
\Diam(\Graph)\approx \Diam(\mG),\qquad 
\rho(\Graph)\approx \rho(\mG),\qquad 
\mGirth(\Graph)= \mGirth(\mG),\qquad 
\]

We fix $d$, hence $p$ and see that along any LPS metric expander
{
\[
\begin{split}
\frac{F_\alpha(\Diam(\Graph^{p,q}), \rho(\Graph^{p,q}), \mGirth(\Graph^{p,q}))}{|\Graph^{p,q}|^{2+\alpha}}&\lesssim \frac{ \log_{d-1}(\#\mV)^\alpha}{ \#\mV^{2+\alpha}} \to 0,\\
{F_{-2}(\Diam(\Graph^{p,q}), \rho(\Graph^{p,q}), \mGirth(\Graph^{p,q}))}&\lesssim  \log_{d-1}(\#\mV)^{-2} \to 0,\\
\end{split}
\qquad \hbox{as }\#\mV\to\infty
\]
in the cases (1) and (2), respectively.
}
In view of \autoref{rem:quantumraman}, this yields the claim.

The case (3) can be studied likewise, observing that along any LPS metric expander
\[
\frac{|\Graph^{p,q}|^{\alpha}}{F_{2+\alpha}(\Diam(\Graph^{p,q}), \rho(\Graph^{p,q}), \mGirth(\Graph^{p,q}))}\gtrsim \frac{ \#\mV^{\alpha}}{ \log_{d-1}(\#\mV)^{2+\alpha}} \to \infty,\qquad\hbox{as }\#\mV\to\infty.
\]
By  \autoref{rem:quantumraman}, this yields the claim.
\end{proof}

\begin{rem}
(1) The claim in \autoref{cor:lowerb} does not contradict with \cite[Théorème 3.1]{Nic87} (existence of a lower bound by $\frac{1}{|\Graph|^2}$), \cite[Theorem~7.1]{KenKurMal16} (existence of an upper bound by $\frac{4|\Graph|-3\Diam(\Graph)}{\Diam(\Graph)^3}$, \cite[Theorem~7.2]{KenKurMal16} (existence of a lower bound by $\frac{1}{2\Diam(\Graph)|\Graph|}$), 
 \cite[Proposition~1.8 and Theorem 1.9]{BerKenKur23} (existence of an upper bound by $\frac{|\Graph|}{\mGirth(\Graph)^3}$ and $\frac{|\Graph|}{\Diam(\Graph)^3}$). Also, it includes the observation in \cite[Example 1.10.(2)]{BerKenKur23} (failure of an upper bound by $\frac{\Diam(\Graph)}{\mGirth(\Graph)^3}$) as a special case.

(2) However, there are bounds that are known to fail -- for instance, the lower bound by $\frac{1}{\Diam(\Graph)^2}$ -- but whose failure does not seem to be explainable by the techniques we used to prove \autoref{cor:lowerb}. For instance, \cite[Example~5.1]{KenKurMal16} aims at keeping a metric quantity (the diameter, in that case) constant along a sequence of metric graphs of increasing volume, whose mean degree grows linearly in the volume while the spectral gap tends to 0.  Indeed, the strategy of the current paper is opposite, as it searches for a sequence of regular graphs with spectral gap bounded below, while the diameter blows up.
\end{rem}

\begin{rem}
There are two different notions of \emph{triameter}:
\begin{itemize}
\item In discrete graph theory, the triameter was introduced in \cite{Das21} as
\[
\Tri(\mG):=\max_{\mv,\mw,\mz}\left(\dist_{\mG}(\mv,\mw)+\dist_{\mG}(\mw,\mz)+\dist_{\mG}(\mz,\mv)\right)
\]
\item In metric graph theory, the {triameter} was introduced in~\cite{BerKenKur23} as
\[
\Tri'(\Graph):=\max_{x,y,z}\min_{\substack{x\ne y\\ y\ne z\\ z\ne x}}\{\dist_\Graph(x,y),\dist_\Graph(y,z),\dist_\Graph(z,x)\}
\]
\end{itemize}
Both notions immediately extend to the other environment, i.e., one can define $\Tri(\Graph)$ and $\Tri'(\mG)$.

It is known that $2\Diam (\mG)\le \Tri(\mG)\le 3\Diam(\mG)$ for any $\mG$, hence $\Tri(\mG)\asymp  \log_{d-1}(\#\mV)$ along expanders. Also, one sees that $\Tri(\Graph)\ge \Tri(\mG)$. This shows that we can extend to $\Tri$ all results we have obtained for $\Diam$ in this article. 

Things are less clear for $\Tri'$, even though it is immediate that $\Tri'(\Graph)\le \Diam(\Graph)$ and $\Tri'(\Graph)\le \Tri(\Graph)$. Indeed, $\Tri'$ is a finer quantity that is strongly affected by the local features of the graph. Deriving its asymptotic behaviour along expanders does not seem to be totally trivial.
\end{rem}

\section{Dirichlet boundary conditions}\label{sec:dirichlet}

Von Below's formula can be extended to metric graph Laplacians $\DeltaGVD $ satisfying Dirichlet conditions on vertex set $\mVD$.
The following characterisation was obtained by Serge Nicaise in \cite[Theorem~3.1]{Nic85}.

\begin{theo}\label{theo:vonbnic}
Let $\Graph$ be a unilateral metric graph, and let 
$\emptyset\ne \mVD\subsetneq \mV$. Then for all $0< \lambda$ such that $\lambda\ne \pi^2k^2$ for any $k\in \N$, $\lambda$ is an eigenvalue of $\DeltaGVD $ if and only if $1-\cos\sqrt{\lambda}$ is an eigenvalue of $\LGVD$.

In particular, 
\begin{equation}\label{eq:nicvonb}
\lambda_1(\Graph;\mVD)=\arccos(1-\nu_1(\mG;\mVD))^2.
\end{equation}
\end{theo}
Here $\LGVD$ is the normalised Laplacian with Dirichlet boundary conditions at $\mVD\subset \mV$: it can be described as follows. Let 
\[
\widetilde{\mV}:=\mV\setminus\mVD,
\]
and let
\[
E:\C^{\widetilde{\mV}}\to \C^{\mV}
\]
be the canonical extension operator by 0.
Now, consider the self-adjoint operator $\LGVD$ on $\ell^2_{\deg}(\mV)$ that is associated with the quadratic form $\mathfrak{a}_\mVD$ given by
\[
\mathfrak{a}_\mVD(f):=\sum_{\me\in \mE}|\nabla E f(\me)|^2,\qquad f:\widetilde{\mV}\to \C,
\]
where we use the notation from \autoref{sec:comblapl}: in particular, $\LGVD=E^*L_\mG E$. Accordingly, and in view of Courant's minmax principle, the lowest eigenvalue of $\LGVD$ is
\begin{equation}\label{eq:courantf-disc-dirich}
\nu_1(\mG;\mVD)=\min_{\substack{f:\tilde{\mV}\to\C}} \frac{\mathfrak{a}_\mVD(f)}{\|f\|^2_{\ell^2_{\deg}(\mV)}}.
\end{equation}

\begin{rem}\label{cor:bifmug-nu1}
A simple lower bound for $\nu_1(\mG;\mVD)$ is the following counterpart of Fiedler's classical bound on $\nu_2(\mG)$ for Dirichlet conditions:
\begin{equation*}\label{eq:courantf-disc-dirich-fiedl}
\nu_1(\mG;\mVD)\ge 2\left(1-\cos\left(\frac{\pi}{2(\#\mV-\#\mVD)-1}\right)\right)
\end{equation*}
It was obtained in \cite[Corollary~5.12]{BifMug25}. 
We may use it to deduce the following from \autoref{theo:vonbnic}:

Let $\mG$ be a graph. Then for the associated unilateral metric graph $\Graph$ there holds
\[
\lambda_1(\Graph;\mVD)\ge \frac{\pi^2}{(2\#\mV-2\#\mVD+1)^2}\qquad\hbox{for any }\emptyset\ne \mVD\subsetneq \mV.
\]
However, this is not really useful for our purposes, since it does not prevent $\lambda_1(\Graph;\mVD)$ from vanishing as $\#\mV$ tend to $\infty$, apart in the trivial case of $\#\mV-\#\mVD={\mathcal O}(1)$.
\end{rem}

The following is an alternative lower bound that is often more convenient for the scope of this paper.

\begin{lemma}\label{lem:perpendic}
If $\mG$ is a $d$-regular, non-bipartite Ramanujan graph, then there holds
\begin{equation*}
\nu_1(\mG;\mVD) 
\ge \left(1 - \frac{2\sqrt{d-1}}{d}\right)
\frac{\#\mVD}{\#\mV}
\qquad\hbox{for any }\emptyset\ne \mVD\subsetneq \mV.
\end{equation*}
\end{lemma}
\begin{proof}
To begin with, let us recall the following well-known fact of spectral graph theory: $0$ is always an eigenvalue of $L_\mG$, with $2$ being an eigenvalue if and only if $\mG$ is bipartite. The corresponding Perron eigenvector is a multiple of $\mathbf{1}$.

For any $u:\widetilde{\mV}\to \C$, we decompose $Eu := u_\parallel + u_\perp$, where $E$ is the canonical extension operator introduced above and 
\[
u_\parallel=\frac{1}{\#\mV}J_\mV Eu=\frac{1}{\#\mV}\sum_{\mv\in\mV}Eu(\mv)\cdot \mathbf{1}
\]
is the projection of $Eu$ onto the $\ker(L_\mG)$ and, consequently, 
\begin{equation}\label{eq:euper-formula}
u_\perp:=\left(\Id_\mV-\frac{1}{\#\mV}J_\mV\right) Eu;
\end{equation}
here, $J_{\mV}$ is the $\#{\mV} \times \#{\mV}$ all-one matrix. 
A direct computation now shows that the norm of $u_\perp$ satisfies
\begin{equation}\label{eq:norm-uperp}
\|u_\perp\|^2 = \left(
 \left( \Id_{\#{\widetilde{\mV}}} - \frac{1}{\#\mV}J_{\widetilde{\mV}} \right) u,u\right)
\end{equation}
where $J_{\widetilde{\mV}}$ is the $\#\widetilde{\mV} \times \#\widetilde{\mV}$ all-one matrix, where we use the fact that $\Id_\mV-\frac{1}{\#\mV}J_\mV$ is an orthogonal projector.


%

By construction, $u_\perp$ is orthogonal to the Perron eigenvector $u_\parallel$,
hence $u_\perp$ lies in the span of all eigenspaces corresponding to non-trivial eigenvalues (this is where the assumption that $\mG$ is non-bipartite is used; cf.~\autoref{rem:eigenv-triv}).
It follows that 
\[
\left|1-\frac{(L_\mG u_\perp,u_\perp)}{\|u_\perp\|^2}\right|\le  \max_{i=2,\ldots,\#\mV} |1-\nu_i(\mG)|,
\]
and therefore, invoking the defining property of Ramanujan graphs,
\begin{equation}\label{eq:auule}
\left|1-\frac{(L_\mG u_\perp,u_\perp)}{\|u_\perp\|^2}\right|\le 
 \frac{2\sqrt{d-1}}{d}.
\end{equation}

%

Now, the normalised Laplacian $L_\mG$ satisfies
\[
(L_\mG Eu,Eu)=(L_\mG u_\parallel,u_\parallel)+2(L_\mG u_\parallel, u_\perp)+(L_\mG u_\perp,u_\perp)=(L_\mG u_\perp,u_\perp)
\]
and by \eqref{eq:auule}
\begin{equation}\label{eq:ineq-hermit-matrices}
\left(1+\frac{2\sqrt{d-1}}{d}\right) \|u_\perp\|^2\ge (L_\mG u_\perp,u_\perp)\ge \left(1-\frac{2\sqrt{d-1}}{d}\right) \|u_\perp\|^2
\end{equation}

Because $\LGVD=E^* L_\mG E$, we conclude that 
\begin{equation}\label{eq:domination-hermit-matrix}
(\LGVD u,u)\ge \left(1-\frac{2\sqrt{d-1}}{d}\right)\left(
 \left( \Id_{\#{\widetilde{\mV}}} - \frac{1}{\#\mV}J_{\widetilde{\mV}} \right) u,u\right).
\end{equation}
Now the lower bound follows immediately, since $\frac{\#\mVD}{\#\mV}$ is the lowest eigenvalue of $ \left( \Id_{\#{\widetilde{\mV}}} - \frac{1}{\#\mV}J_{\widetilde{\mV}} 
\right)$.
\end{proof}

\begin{rem}\label{rem:upper-raman-nu1}
Again by \eqref{eq:auule}, the same computation yields
\begin{equation}\label{eq:domination-hermit-matrix-bis}
(\LGVD u,u)\le \left(1+\frac{2\sqrt{d-1}}{d}\right)\left(
 \left( \Id_{\#{\widetilde{\mV}}} - \frac{1}{\#\mV}J_{\widetilde{\mV}} \right) u,u\right)\qquad \hbox{for all }u:\widetilde{\mV}\to\C
\end{equation}
and, in particular,
\begin{equation}\label{eq:domination-hermit-matrix-eig-ter}
\nu_1(\mG;\mVD)\le \nu_{\#\widetilde{\mV}}(\mG;\mVD)\le 1+\frac{2\sqrt{d-1}}{d},
\end{equation}
because the largest eigenvalue of $\Id_{\#{\widetilde{\mV}}} - \frac{1}{\#\mV}J_{\widetilde{\mV}}$ is always $1$. Indeed, the $i$-th eigenvalue of $\Id_{\#{\widetilde{\mV}}} - \frac{1}{\#\mV}J_{\widetilde{\mV}}$ is $1-\frac{1}{\#\mV}{\mu_i}(J_{\widetilde{\mV}})$, where $\mu_i(J_{\widetilde{\mV}})$ is the $i$-th eigenvalue of $J_{\widetilde{\mV}}$; and the lowest eigenvalue of $J_{\widetilde{\mV}}$ is 0, with multiplicity $\#\widetilde{\mV}-1$.
\end{rem}

In view of \eqref{eq:nicvonb} and the elementary inequality 
\[
\arccos(1-x)^2 \ge 2x+\frac{x^3}{72}\qquad\hbox{for all }x \in [0,1],
\] 
combining \eqref{eq:nicvonb} and  \autoref{lem:perpendic} we can finally state the following.

\begin{cor}\label{cor:lambda1-ramanu}
If $\mG$ is a $d$-regular, non-bipartite Ramanujan graph, then for the associated unilateral metric graph $\Graph$ there holds
\begin{equation}\label{eq:eigenvaluelowerbound}
\lambda_1(\Graph;\mVD)  \ge 2\left(1 - \frac{2\sqrt{d-1}}{d}\right)\frac{\#{\mVD}}{\#\mV}\qquad\hbox{for any }\emptyset\ne \mVD\subsetneq \mV.
\end{equation}
\end{cor}

This lower bound is akin to \eqref{eq:low-bound-lambda2} and will analogously help us prove the failure of certain upper bounds.

\subsection{Torsional rigidity}\label{sec:torsional}
The \emph{torsion function} of a metric graph $\Graph$ with Dirichlet boundary conditions at $\mVD$ is the solution of the Poisson problem
\[
-\DeltaGVD u(x)=1,\qquad x\in \Graph,
\]
where $u$ necessarily satisfies Dirichlet conditions at $\mVD$ by construction of $\DeltaGVD$; 
the \emph{torsional rigidity} $T(\Graph;\mVD)$ is then the $L^1$-norm of the torsion function. Remarkably,
\[
T(\Graph;\mVD)=\max_{u\in H^1_0(\Graph;\mVD)} \frac{\|u\|^2_{L^1(\Graph)}}{{a}(u)}.
\]

Several bounds on $T(\Graph;\mVD)$ are known, including the \emph{Pólya--Szeg\H{o} inequality}
\begin{equation}\label{eq:polyasz}
\lambda_1(\Graph;\mVD)T(\Graph;\mVD)<|\Graph|.
\end{equation}
that was proved in~\cite[Proposition~5.1]{MugPlu23}.
The proof is based on an apparently lossy estimate, and so it was speculated in \cite[Example~5.4]{MugPlu23} that this estimate may remain true upon replacing the right-hand side of \eqref{eq:polyasz} by $\frac{\pi^2}{12}|\Graph|$. Let us disprove this conjecture.

\begin{theo}\label{rem:no-polya-szego-improved}
Let $\emptyset\ne \mVD\subsetneq \mV$. 
Then there is \underline{no} universal constant $C<1$ such that
\[
\lambda_1(\Graph;\mVD)T(\Graph;\mVD)\le C |\Graph|.
\]
for all unilateral metric graphs.
\end{theo}

\begin{proof}
Throughout this proof, $\Graph$ will be $d$-regular unilateral metric graphs constructed upon a non-bipartite LPS expander, $\mathcal G=\mathcal{G}^{p,q}$. 

We introduce the torsional rigidity $T(\mG;\mVD)$ of $\mG$ as the $\ell^1$-norm of the torsion function of $\mG$, i.e., of the unique solution $u$ of
\[
\LGVD u=\mathbf{1}.
\]
It is known that the torsional rigidity for $\Graph$ and $\mG$ are related by a very explicit formula: the Handshaking Lemma yields $|\Graph|=\frac{d\#\mV}{2}$ and, therefore, the formula in~\cite[Theorem~3.9]{MugPlu23} simplifies to 
\begin{equation}\label{eq:discretetorsion}
T(\Graph;\mVD)= \frac{d\#\mV}{24}+\frac{d}{4}\left(\LGVD^{-1}\mathbf{1}_{\widetilde{\mV}},\mathbf{1}_{\widetilde{\mV}}\right).
\end{equation}

Let us search for a lower bound of \eqref{eq:discretetorsion}:
Combining \eqref{eq:norm-uperp} and \eqref{eq:ineq-hermit-matrices} (here we use non-bipartiteness of the LPS expander), we see that
\begin{equation*}
\LGVD \le \left(1 + \frac{2\sqrt{d-1}}{d}\right)\left(\Id_{\#{\widetilde{\mV}}} - \frac{1}{\#\mV}J_{\widetilde{\mV}}\right)
\end{equation*}
in the sense of Hermitian matrices.
We thus obtain
\begin{equation}
\LGVD^{-1} \ge \left( \frac{d}{d+2\sqrt{d-1}}\right)\left(\Id_{\#{\widetilde{\mV}}} +\frac{1}{\#\mVD}J_{\widetilde{\mV}}\right)
\end{equation}
since $\Id_{\#{\widetilde{\mV}}} - \frac{1}{\#\mV}J_{\widetilde{\mV}}$ is invertible with inverse 
$\Id_{\#{\widetilde{\mV}}} + \frac{1}{\#\mVD}J_{\widetilde{\mV}}$: indeed,
\[
\begin{split}
\left(\Id_{\#{\widetilde{\mV}}} - \frac{1}{\#\mV}J_{\widetilde{\mV}}\right)\left(\Id_{\#{\widetilde{\mV}}} + \frac{1}{\#\mVD}J_{\widetilde{\mV}}\right)
&=\Id_{\#{\widetilde{\mV}}}- \frac{1}{\#\mV}J_{\widetilde{\mV}}+ \frac{1}{\#\mVD}J_{\widetilde{\mV}}-\frac{1}{\#\mV\#\mVD}J_{\widetilde{\mV}}J_{\widetilde{\mV}}\\
&=\Id_{\#{\widetilde{\mV}}}+\frac{\#\mV-\#\mVD-\#\widetilde{\mV}}{\#\mV\#\mVD} J_{\widetilde{\mV}}=\Id_{\#{\widetilde{\mV}}}
\end{split}
\]
recalling that, by construction, $\widetilde{\mV}=\mV\setminus \mVD$.
We conclude that
\begin{equation}
\left(\LGVD^{-1} \mathbf{1}_{\#{\widetilde{\mV}}},\mathbf{1}_{\#{\widetilde{\mV}}}\right) \ge \frac{d}{d + 2\sqrt{d-1}} \left( \left(\Id_{\#{\widetilde{\mV}}} + \frac{1}{\#\mVD}J_{\widetilde{\mV}}\right) \mathbf{1}_{\#{\widetilde{\mV}}} ,\mathbf{1}_{\#{\widetilde{\mV}}}\right)= \frac{d}{d + 2\sqrt{d-1}}\cdot \frac{\#\mV \#\widetilde{\mV}}{\#\mVD}.
\end{equation}
Substituting this into \eqref{eq:discretetorsion} yields 
\begin{equation}\label{eq:torsionlowerbound}
T(\Graph;\mVD) \ge  \frac{d\#\mV}{24}+\frac{d^2\#\mV \#\widetilde{\mV}}{4(d + 2\sqrt{d-1})\#\mVD}.
\end{equation}

We now let $\mVD$ be an arbitrary singleton.
Combining \eqref{eq:torsionlowerbound} with \eqref{eq:eigenvaluelowerbound} and  $|\Graph| = \frac{d\#\mV}{2}$ we finally obtain, for any Ramanujan metric graph $\Graph$ {in which $\mVD$ is a singleton,}
\begin{equation}
\begin{split}
\frac{\lambda_1(\Graph;\mVD) T(\Graph;\mVD)}{|\Graph|} 
&\ge \frac{2\left(1 - \frac{2\sqrt{d-1}}{d}\right)\frac{\#\mVD}{\#\mV} \left[ \frac{d\#\mV}{24}+\frac{d^2\#\mV \#\widetilde{\mV}}{4(d + 2\sqrt{d-1})\#\mVD} \right]}{\frac{d\#\mV}{2}}\\
& = \frac{1}{6\#\mV}\left(1 - \frac{2\sqrt{d-1}}{d}\right) + \frac{\#{\mV}-1}{\#\mV}\cdot\frac{d - 2\sqrt{d-1}}{d + 2\sqrt{d-1}}.
\end{split}
\end{equation}

To prove that the supremum over all possible metric graphs is exactly $1$, we evaluate a double asymptotic limit over a \emph{non-bipartite} LPS expander $(\mG^{p,q})$ and the associated LPS metric expander $(\Graph^{p,q})$. As usual, $(p,q)$ is taken to be a pair of compatible primes in the LPS construction, but this time we want to study the asymptotics behavior as both $p,q$, hence both $d,\#\mV$ grow.
In this way we find
\begin{equation}
\begin{split}
\frac{\lambda_1(\Graph;\mVD) T(\Graph;\mVD)}{|\Graph|} 
&\ge 1-o(\#\mV^{-1})-o(d^{-\frac{1}{2}})\qquad \hbox{as }d,\#\mV\to\infty.
\end{split}
\end{equation}

%
This finally proves the claim.
\end{proof}

\subsection{Inradius and mean distance to $\mVD$}\label{sec:meandistdir}

We conclude this note by presenting a case where it seems that our approach cannot be made to work. Let us in the following use the following notation: We write
\begin{equation}\label{eq:eta-doubly}
\eta:=\begin{cases}
2\quad\hbox{if $\Graph$ is doubly connected upon gluing all vertices in $\mVD$},\\
1\quad\hbox{otherwise}.
\end{cases}
\end{equation}

In $\Graph$, the \emph{inradius} and the \emph{mean distance} from the Dirichlet set $\mVD$ are defined as
\begin{equation}\label{eq:inrmean}
\Inr(\Graph;\mVD):=\max_{x\in \Graph}\dist_\Graph(x,\mVD)\qquad\hbox{and}\qquad
\rho(\Graph;\mVD):=\frac{1}{|\Graph|}\int_\Graph \dist_\Graph(x,\mVD)\dd x,
\end{equation}
respectively. Likewise, in $\mG$ one can define the inradius and
 the mean distance from $\mVD$
\begin{equation}
\Inr(\mG;\mVD) := \max_{\mv\in \mG}\dist_{\mG}(\mv,\mVD)\quad\hbox{and}\quad \rho(\mG;\mVD) := \frac{1}{\#\mV} \sum_{\mv\in\mV} \dist_{\mG}(\mv,\mVD).
\end{equation}

Observe that
\begin{equation}\label{eq:rhoinrV}
\rho(\Graph;\mVD)<\Inr(\Graph;\mVD)\le \frac{|\Graph|}{\eta}
\end{equation}
where the first inequality follows from Hölder and the second can be proved adapting~\cite[Lemma~6.2]{BapKenMug24} (whose proof also shows that the second inequality is strict unless $\Graph$ is a caterpillar graph ($\eta=2$) or an interval with Dirichlet conditions at precisely one endpoint $(\eta=1$)).
Also, it was proved in~\cite[Theorem~5.1]{DelRos16} that
\[
\frac{1}{\Inr(\Graph;\mVD)}=\lim_{p\to\infty} \inf_{u\in W^{1,p}_0(\Graph;\mVD)} \frac{\|u'\|_p}{\|u\|_p}.
\]

Both quantities in~\eqref{eq:inrmean} were extensively studied in \cite[Section~4.4]{Plu21b}, where it was proved that there is \underline{no} universal constant $c>0$ such that
\[
\frac{c}{\Inr(\Graph;\mVD)^2}\le \lambda_1(\Graph;\mVD)
\]
for all unilateral metric graphs, see \cite[Example~4.4.4]{Plu21b}; but that the uniform lower bound
\[
\frac{1}{|\Graph|\Inr(\Graph;\mVD)}\le \frac{1}{|\Graph|\rho(\Graph;\mVD)}\le\lambda_1(\Graph;\mVD)
\]
does hold for all metric graphs, see \cite[Theorem~4.4.3]{Plu21b}. (In fact, even the much subtler bound
\[
\frac{1}{\sup_{x\in \Graph}|B(x;\Inr(\Graph;\mVD))| \Inr(\Graph;\mVD)}\le\lambda_1(\Graph;\mVD),
\]
whose proof is based on Zorn's Lemma, holds \cite[Theorem~4.5.11]{Plu21b}.)

We also mention the upper bound
\begin{equation}\label{eq:triv-rho}
\lambda_1(\Graph;\mVD)<\frac{1}{\rho(\Graph;\mVD)^2}
\end{equation}
that has been obtained by a test function argument in~\cite[Lemma~4.1]{BapKenKra26}. Like all upper estimates based on test function arguments, it seems unlikely that the estimate in \eqref{eq:triv-rho} is sharp. It is thus natural to formulate the following.

\begin{conj}\label{rem:no-mvd-bound}
Let $\emptyset\ne \mVD\subsetneq \mV$. 
Then there exists a universal constant $C\in (0,1)$ such that
\[
\lambda_1(\Graph;\mVD)\le \frac{C}{\rho(\Graph;\mVD)^2}
\]
for all metric graphs.
\end{conj}

On the other hand, can the stronger bound 
\[
\lambda_1(\Graph;\mVD)\lesssim \frac{1}{\Inr(\Graph;\mVD)^2}
\]
hold? Let us prove that this is not the case.

\begin{theo}\label{theo:no-inr-bound}
Let $\emptyset\ne \mVD\subsetneq \mV$. 
Then there is \underline{no} universal constant $C>0$ such that
\begin{equation}\label{eq:upper-est-inr}
\lambda_1(\Graph;\mVD)\le \frac{C}{\Inr(\Graph;\mVD)^2}
\end{equation}
either for all unilateral bipartite metric graphs, or for all unilateral non-bipartite metric graphs.
\end{theo}

The proof will be divided into two parts: the first one is based on expanders and shows us that \eqref{eq:upper-est-inr} cannot hold for all \emph{non-bipartite} unilateral metric graphs; it is crucially based on \autoref{prop:positive-density-large-inradius} below.
The second one, kindly suggested to us by ChatGPT 5.6 pro as prompted by Robert J.\ George \cite{Geo26}, is similar in spirit and shows that \eqref{eq:upper-est-inr} cannot hold for all \emph{bipartite} unilateral metric graphs. (We actually present here a simplified version of ChatGPT's proof.)


%
%

%

\begin{lemma}\label{prop:positive-density-large-inradius}
Let $d\in \N$ be fixed, $d\ge 3$, and let \((\mG_n)_{n\in\N}\) be a sequence of \(d\)-regular graphs with strictly increasing number of vertices. Then there exist subsets \(\mVDn\subset \mV_n\) such that
\begin{equation}\label{eq:density-to-one}
\lim_{n\to \infty}\frac{\#\mVDn}{\#\mV_n}= 1
\end{equation}
and, for the associated unilateral metric graphs \(\Graph_n\),
\begin{equation}\label{eq:inradius-log-lower}
\left\lfloor \frac{1}{2}\log_{d-1}\#\mV_n\right\rfloor+1\le \Inr(\mG_n;\mVDn)\qquad\hbox{as }n\to \infty.
\end{equation}
\end{lemma}

\begin{proof}
In the following, for any \(\mv\in \mV\) and any $r>0$, denote by
\[
B^\mG_r(\mv)
:=
\{\mw\in \mV:\dist_{\mG}(\mv,\mw)\le r\}
\]
the combinatorial ball in $\mG$ of radius \(r\) centered at \(\mv\). Observe that for any $d\ge 3$, any $d$-regular graph $\mG$,  and any $\mv\in \mV$ the Moore bound yields
\begin{equation}\label{eq:moore-radius}
\#B^{\mG}_{r}(\mv)
\le
1+d\sum_{j=0}^{r-1}(d-1)^j
=
1+d\,\frac{(d-1)^r-1}{d-2}\le C_d (d-1)^r
\end{equation}
for some constant \(C_d>0\) that depends only on \(d\).

Let us now start the construction of the sets $\mVDn$: for each \(n\in \N\), choose a vertex \(\mv_n\in \mV_n\) and set
\[
r_n:=\left\lfloor \frac12 \log_{d-1}\#\mV_n\right\rfloor,
\]
whence by~\eqref{eq:moore-radius}
\[
\#B^{\mG_n}_{r_n}(\mv_n)
\le
C_d (d-1)^{r_n}
\le
C_d \#\mV_n^\frac12.
\]
Let now
\[
\mVDn:=\mV_n\setminus B^{\mG_n}_{r_n}(\mv_n)
\]
and observe that 
\[
\frac{\#\mVDn}{\#\mV_n}
=
1-\frac{\#B^{\mG_n}_{r_n}(\mv_n)}{\#\mV_n}
\ge
1-C_d \#\mV_n^{-\frac12}
\to 1\qquad\hbox{as }n\to\infty,
\]
which proves \eqref{eq:density-to-one}. In particular,  $\mVDn\ne \emptyset$ for all $n$ large enough.

By construction, every vertex in \(\mVDn\) has distance at least \(r_n+1\) from \(\mv_n\); and since \(\mVDn\neq\emptyset\), this lower bound for the distance must be attained.
In other words,
\[
\Inr(\mG_n;\mVDn)\ge \dist_{\mG_n}(\mv_n;\mVDn)
=
r_n+1,
\]
thus completing the proof of \eqref{eq:inradius-log-lower}.
\end{proof}


\begin{proof}[Proof of \autoref{theo:no-inr-bound}]
We first prove the assertion in the non-bipartite case.

Let $(G_n)$ be a family of non-bipartite Ramanujan graphs of fixed degree $d\ge 3$ with strictly increasing number of vertices. We consider the sequence of corresponding Ramanujan metric graphs $(\Graph_n)$.
By \autoref{prop:positive-density-large-inradius} we find Dirichlet sets \(\mVDn\subset\mV_n\) such that \eqref{eq:density-to-one} and \eqref{eq:inradius-log-lower} hold.
Because, clearly, $\Inr (\Graph;\mVD)\ge \Inr (\mG;\mVD)$ for any combinatorial graph $\mG$ and the associated unilateral metric graph $\Graph$, by \eqref{eq:inradius-log-lower} also
\begin{equation}\label{eq:inr-blows-metric}
\Inr(\Graph_n;\mVDn)
\ge
\left\lfloor \frac12\log_{d-1}\#\mV_n\right\rfloor+1
\to\infty\qquad\hbox{as }n\to\infty.
\end{equation}
By \autoref{cor:lambda1-ramanu} and~\eqref{eq:density-to-one}, there exists \(n_0\) such that
\begin{equation}\label{eq:lambda1-bdd-below-q}
\lambda_1(\Graph_n;\mVDn)
\ge
1-\frac{2\sqrt{d-1}}{d}\qquad\hbox{for all }n\ge n_0.
\end{equation}
Combining \eqref{eq:lambda1-bdd-below-q} and~\eqref{eq:inr-blows-metric} yields the claim.
%

Let us now turn to the bipartite case.
To this aim, let again $d\ge 3$, and let $\mT^{(d)}$ be the infinite $d$-regular tree rooted at a vertex $\mv_0$.
By Kesten's formula, see e.g.~\cite[Theorem~12.6.2]{Dav07}, the adjacency spectrum of $\mT^{(d)}$ is $[-2\sqrt{d-1},2\sqrt{d-1}]$; hence, the spectrum of $\mathcal L=\Id-d^{-1}A$ is $[1-\frac{2}{d}\sqrt{d-1},1+\frac{2}{d}\sqrt{d-1}]$.

Now, let us consider the ball $\mT^{(d)}_r:=B^{\mT^{(d)}}_r(\mv_0)$ within $\mT^{(d)}$ centered at $\mv_0$; and impose Dirichlet boundary conditions at all leaves of $\mT^{(d)}_r$, i.e., $\mVD=\partial \mT^{(d)}_r$: we apply domain monotonicity to conclude that
\begin{equation}\label{eq:no-inr-upp-bd-chatgp}
\nu_1(\mT^{(d)}_r;\mVD)\ge \nu_1(\mT^{(d)})\ge 1-\frac{2\sqrt{d-1}}{d}\qquad\hbox{for all }r>0.
\end{equation}
By \autoref{theo:vonbnic} the corresponding unilateral metric graph $\mathcal T^{(d)}_r$ satisfies
\[
\lambda_1(\mathcal T^{(d)}_r;\mVD	)\ge \arccos\left(\frac{2\sqrt{d-1}}{d}\right)^2\qquad\hbox{for all }r>0.
\]
This immediately yields the claim, since $\Inr(\mathcal T^{(d)}_r;\mVD)=r$.
\end{proof}

Let us conclude by observing that, in spite of \autoref{theo:no-inr-bound}, an upper bound is possible if inradius and mean distance are combined.

\begin{prop}\label{lemma:triv-rhoinr}
Let $\emptyset\ne \mVD\subsetneq \mV$. Then for any metric graph $\Graph$ there holds
\begin{equation}\label{eq:l1inrrho}
\lambda_1(\Graph;\mVD)<\frac{\#\mE+1}{\Inr(\Graph;\mVD)\rho(\Graph;\mVD)};
\end{equation}
and, in fact,
\[
\lambda_1(\Graph;\mVD)<\frac{\#\mE+2}{2\Inr(\Graph;\mVD)\rho(\Graph;\mVD)}
\]
if, additionally, $\Graph$ is doubly connected upon gluing all vertices in $\mVD$.
\end{prop}

\begin{proof}
We adapt the proof of \cite[Lemma~4.1]{BapKenKra26} to find
\begin{equation}\label{eq:new-interp}
\lambda_1(\Graph;\mVD)\le \frac{\|\dist_\Graph(\cdot;\mVD)'\|^2_2}{\|\dist_\Graph (\cdot;\mVD)\|^2_2}<\frac{(\#\mE+1)|\Graph|}{\|\dist_\Graph (\cdot;\mVD)\|_1\|\dist_\Graph (\cdot;\mVD)\|_\infty}
= \frac{\#\mE+1}{\Inr(\Graph;\mVD)\rho(\Graph;\mVD)},
\end{equation}
and likewise 
\[
\lambda_1(\Graph;\mVD)< \frac{\#\mE+2}{2}\cdot \frac{1}{\Inr(\Graph;\mVD)\rho(\Graph;\mVD)},
\]
in the doubly connected case, where the second inequality in~\eqref{eq:new-interp} follows from \autoref{lem:dist-gm} below.
\end{proof}

In the proof of \autoref{lemma:triv-rhoinr} we have used the following reverse interpolation inequality, which may be of independent interest.

\begin{lemma}\label{lem:dist-gm}
Let $\mathcal{G}$ be a compact metric graph, and $\emptyset\ne \mVD \subsetneq \mV$. For $\eta$ defined in~\eqref{eq:eta-doubly}, let $c_\eta$ be the unique real root of
\[
\Theta_\eta(x):=    (2\#\mE-\eta )x^3 + 3\eta x - 2\eta,\qquad \eta=1,2.
\]

Then the function $f:\Graph\ni x\mapsto \dist_\Graph (x, \mVD)\in [0,\infty)$ satisfies
\begin{equation}\label{eq:f1f2finty}
\|f\|_2^2 \ge c_1 \|f\|_1 \|f\|_\infty;
\end{equation}
and, in fact,
\begin{equation}\label{eq:f1f2finty-doubly}
\|f\|_2^2 \ge c_2 \|f\|_1 \|f\|_\infty
\end{equation}
if, additionally, $\Graph$ is doubly connected upon gluing all vertices in $\mVD$, 
and there holds
\[
c_1 \ge \frac{1}{\#\mE+1}\qquad\hbox{and}\qquad c_2\ge \frac{2}{\#\mE+2}.
\]
Indeed, $c_\eta\asymp \left(\frac{\eta}{\#\mE}\right)^{\frac13}$ as $\#\mE\to\infty$.
\end{lemma}

Our proof is partially inspired by the proofs of \cite[Lemma~3.7]{BerKenKur17} and \cite[Theorem~4.6]{MugPlu23}.

\begin{proof}
Let $M := \|f\|_\infty$. Also, $f(\mVD)=0$.
On the one hand, because $f$ is continuous over $\Graph$, and by the Intermediate Value Theorem, $[0,M]\subset f(\Graph)$; hence, $f^{-1}(t)$ contains, for a.e.\ $t \in [0, M]$, at least one element; and indeed, at least two if $\Graph$ is doubly connected upon gluing all vertices in $\mVD$.

On the other hand, since $f$ is a distance function on a metric graph, $|f'(x)| = 1$ for a.e.~$x\in \Graph$. Hence, on any single edge, $f(x)$ is the minimum of at most two affine functions of slope $\pm 1$, meaning it can achieve a given value $t$ at most twice; or at most $2\#\mE$ times over the whole $\Graph$. Let us introduce
\[
\Graph(t):= f^{-1}(t)=\{ x \in \mathcal{G} \mid f(x) = t \}
\]
and
\[
P(t) := \#\Graph(t):
\]
we have just observed that for a.e.\ $t\in [0,M]$,
\[
2\#\mE\ge P(t)\ge \left\{
\begin{aligned}
&1\quad\hbox{for general }\Graph,\\
&2\quad \hbox{in the doubly connected case.}
\end{aligned}
\right.
\]
By the coarea formula we see that for any $p\in [1,\infty)$ and $h(t):=t^p$, $t\in (0,\infty)$
\[
\int_\Graph h(f(x))\dd x= 
\int_0^M h(t) P(t) \dd t
    \]
 i.e.,
 \[
     \|f\|_p^p =      \int_0^M t^p P(t) \dd t,
     \]
see~\cite[Lemma~6.1]{BapKenMug24} for the case $p=1$.

In order to prove \eqref{eq:f1f2finty} and \eqref{eq:f1f2finty-doubly}, we seek to minimise the ratio
\[
    R := \frac{\|f\|_2^2}{\|f\|_\infty \|f\|_1} = \frac{\int_0^M t^2 P(t) \dd t}{M \int_0^M t P(t) \dd t};
\]
or else
\[
    R := \frac{\int_0^1 s^2 Q(s) \dd s}{\int_0^1 s Q(s) \dd s}
\]
over all measurable functions $Q: [0, 1] \to [\eta, 2\#\mE]$, as one  sees substituting $s :=\frac{t}{M}$ and defining $Q(s) := P(Ms)$, $s\in [0,1]$; here $\eta$ is defined as in~\eqref{eq:eta-doubly}.

The measure $s^2 \dd s$ in the numerator decays faster near zero than the measure $s \dd s$ in the denominator, the ratio $R$ is minimised when $Q(s)$ is as large as possible for small $s$ and as small as possible for large $s$. Thus, by the bathtub principle the minimiser of $R$ is a step function of the form
\[
    Q(s) = \begin{cases} 2\#\mE & \text{if } 0 \le s \le x, \\ \eta & \text{if } x < s \le 1, \end{cases}
\]
for some optimal threshold $x \in [0, 1]$; i.e., we aim to minimise
\[
R(x):=\frac{ \int_0^x 2\#\mE s^2\dd s + \eta\int_x^1  s^2\dd s }{ \int_0^x 2\#\mE s\dd s + \eta\int_x^1  s\dd s}
\]
over all $x\in [0,1]$.

Let us first focus on the case of general graphs. 
For this family of step functions, we find
\[
    \int_0^1 s Q(s) \dd s = \int_0^x 2\#\mE s\dd s + \int_x^1 s\dd s = \#\mE x^2 + \frac{\eta}{2}(1 - x^2) = \frac{1}{2}\big(\eta + (2\#\mE-\eta)x^2\big),
\]
and
\[
    \int_0^1 s^2 Q(s) \dd s = \int_0^x 2\#\mE s^2\dd s + \eta\int_x^1 s^2\dd s = \frac{2\#\mE}{3}x^3 + \frac{\eta}{3}(1 - x^3) = \frac{1}{3}\big(\eta + (2\#\mE-\eta)x^3\big).
\]
The ratio as a function of $x$ becomes
\[
    R(x) = \frac{2}{3} \cdot\frac{\eta + (2\#\mE-\eta)x^3}{\eta + (2\#\mE-\eta)x^2},\qquad x\in [0,1].
\]
To minimise $R(x)$, we look for  $x_0$ such that  $R'(x_0)=0$, i.e., for a root of the polynomial
\begin{equation}\label{eq:polyn-g}
\Theta(x):=    (2\#\mE-\eta)x^3 + 3\eta x - 2\eta.
\end{equation}
Observe that 
\[
\Theta'(x)=3\left(\left(2\#\mE-\eta\right)x^2+\eta\right)>0\qquad\hbox{for all }x\in (0,1),
\]
 i.e., $\Theta$ is a strictly monotonically increasing function such that $\Theta(0)=-2\eta$ and $\Theta(1)=2\#\mE-1+3-2=2\#\mE>0$: therefore, there exists a unique real root $x_0\in (0,1)$ of $\Theta$. 

Substituting back into $R$ we find
\[
    R(x_0)
     = \frac{2}{3}\cdot \frac{\eta + \eta\left(\frac{2-3x_0}{x_0^3}\right)x_0^3}{\eta + \eta \left(\frac{2-3x_0}{x_0^3}\right)x_0^2} 
    = \frac{2}{3}\cdot \frac{3x_0^3 -3x_0^4 }{-2x_0^3 +2x_0^2} 
    = x_0.
\]
That is, the minimum of $R$ is achieved at the real root $x_0$ of \eqref{eq:polyn-g}. In other words,
\[
 \frac{\|f\|_2^2}{\|f\|_\infty \|f\|_1} = \frac{\int_0^M t^2 P(t) \dd t}{M \int_0^M t P(t) \dd t}\ge x_0.
\]

Finally, let us show that $x_0>\frac{\eta}{\#\mE+\eta}$: because $x_0\in (0,1)$ and, hence, $x_0^3<x_0$, 
\[
\Theta(x)= (2\#\mE-\eta)x^3 + 3\eta x - 2\eta < (2\#\mE-\eta)x + 3\eta x - 2\eta = (2\#\mE + 2\eta)x - 2\eta\qquad\hbox{for all }x\in (0,1),
\]
and in particular
\[
0=\Theta(x_0)<(2\#\mE + 2\eta)x_0 - 2\eta,
\]
whence the inequalities \eqref{eq:f1f2finty} and~\eqref{eq:f1f2finty-doubly} follow.

 Likewise
\[
\Theta(x)= (2\#\mE-\eta)x^3 + 3\eta x - 2\eta > (2\#\mE-\eta)x^3 + 3\eta x^3 - 2\eta = (2\#\mE + 2\eta)x^3 - 2\eta\qquad\hbox{for all }x\in (0,1),
\]
and in particular
\[
0=\Theta(x_0)>(2\#\mE + 2\eta)x_0^3 - 2\eta,
\]
i.e.,
\[
\sqrt[3]{\frac{2\eta}{2\#\mE+2\eta}}>x_0.
\]
This shows that $x_0\searrow 0$ as $\#\mE\to\infty$.
The claimed asymptotic behaviour of $c_\eta$ follows balancing the cubic term against the constant term in $\Theta_\eta(x)=0$. 
\end{proof}

\begin{rem}
It is possible to prove that, with the same notation as in \autoref{lem:dist-gm}, there holds  for all $\Graph$
\[
\|f\|_2^2 \ge \frac{1}{3(2\#\mE)^\alpha}\|f\|_1^{\alpha} \|f\|^{3-2\alpha}_\infty,
\]
for any $\alpha\in (0,1]$, whence adapting the proof of \autoref{lemma:triv-rhoinr} we see that \eqref{eq:l1inrrho} can be, for all $\alpha\in (0,1]$ and some $c=c(\alpha,\#\mE)$, extended to
\[
\lambda_1(\Graph;\mVD)\le\frac{c(\alpha,\#\mE)|\Graph|^{1-\alpha}}{\Inr(\Graph;\mVD)^{3-2\alpha}\rho(\Graph;\mVD)^\alpha}.
\]
\end{rem}

\end{document}